\documentclass[a4paper,11pt,reqno]{amsart}
\usepackage{graphicx}
\usepackage{a4wide}
\usepackage{amssymb}
\usepackage{amsthm}
\usepackage{eucal}
\usepackage[pdftex,colorlinks]{hyperref}

\numberwithin{equation}{section}

\newcommand{\R}{\mathbb R}
\newcommand{\C}{\mathbb C}
\newcommand{\N}{\mathbb N}
\newcommand{\Z}{\mathbb Z}
\newcommand{\Zs}{\Z \setminus \{ 0\}}
\newcommand{\T}{\mathbb T}

\newcommand{\cH}{\mathcal H}
\newcommand{\be}{\begin{equation}}
\newcommand{\ee}{\end{equation}}
\newcommand{\ba}{\begin{eqnarray}}
\newcommand{\ea}{\end{eqnarray}}

\newcommand{\lan}{\langle}
\newcommand{\ran}{\rangle}

\newtheorem{theorem}{Theorem}[section]
\newtheorem{proposition}[theorem]{Proposition}

\newtheorem{lemma}[theorem]{Lemma}
\newtheorem{corollary}[theorem]{Corollary}
\newtheorem{definition}[theorem]{Definition}

\begin{document}

\title[Control of the wave equation with structural damping]{Null controllability of the structurally damped wave equation with moving point control}
\author{Philippe Martin}
\address{Centre Automatique et Syst\`emes\\
Mines ParisTech\\
60 boulevard Saint-Michel\\
75272 Paris Cedex, France }
\email{philippe.martin@mines-paristech.fr}

\author{Lionel Rosier}
\address{Institut Elie Cartan, UMR 7502 UHP/CNRS/INRIA,
B.P. 70239, 54506 Vand\oe uvre-l\`es-Nancy Cedex, France}
\email{rosier@iecn.u-nancy}

\author{Pierre Rouchon}
\address{Centre Automatique et Syst\`emes\\
Mines ParisTech\\
60 boulevard Saint-Michel\\
75272 Paris Cedex, France }
\email{pierre.rouchon@mines-paristech.fr}



\begin{abstract}
We investigate the internal controllability of the wave equation with structural damping
on the one dimensional torus. We assume that the control is acting on a moving point  or on a moving small interval with a constant velocity.
  We prove that the null controllability holds in some suitable Sobolev space  and after a fixed positive time independent of the initial conditions.
\end{abstract}

\maketitle
\section{Introduction}

In this paper we consider the wave equation with {\em structural damping}\footnote{The terminology {\em internal damping} is also used by some authors.}
\be
y_{tt}-y_{xx}-\varepsilon y_{txx}=0
\label{A1}
\ee
where $t$ is time, $x\in \T=\R / (2\pi \Z)$ is the space variable, and $\varepsilon$ is a small positive parameter
corresponding to the strength of the structural damping. That equation has been proposed in \cite{PSM}
as an alternative model for the classical spring-mass-damper PDE.
We are interested in the control properties of \eqref{A1}. The exact controllability of \eqref{A1} with an internal control function supported
in the {\em whole} domain  was studied in \cite{LT,Leugering}. With a boundary control, it was proved in \cite{RR} that \eqref{A1} is not spectrally controllable
(hence not null controllable), but that some approximate controllability may be obtained in some appropriate functional space.

The bad control properties from \eqref{A1} come from the existence of a finite accumulation point in the spectrum. Such a phenomenon was noticed
first by D. Russell in \cite{Russell} for the beam equation with internal damping, by G. Leugering and E. J. P. G. Schmidt in \cite{LS} for the plate equation with internal
damping, and by S. Micu in \cite{Micu} for the linearized Benjamin-Bona-Mahony (BBM) equation
\be
\label{A2}
y_t + y_x -y_{txx}=0.
\ee
Even if the BBM equation arises in a quite different physical context, its control properties share important common features with \eqref{A1}. Remind first that the full BBM equation
\be
\label{A3}
y_t + y_x -y_{txx} + yy_x=0
\ee
is a popular alternative to the Korteweg-de Vries (KdV) equation
\be
\label{A4}
y_t+y_x+y_{xxx}+yy_x=0
\ee
as a model for the propagation of unidirectional small amplitude long water waves in a uniform channel. \eqref{A3} is often obtained from \eqref{A4} in the derivation
of the surface equation by noticing that, in the considered regime, $y_x\sim -y_t$, so that $y_{xxx}\sim -y_{txx}$.  The dispersive term $-y_{txx}$ has a strong smoothing effect,
thanks to which the wellposedness theory of \eqref{A3} is dramatically easier than for \eqref{A4}. On the other hand, the control properties of \eqref{A2} or \eqref{A3}
are very bad (compared to those of \eqref{A4}, see \cite{RZ2009})  precisely because of that term.
It is by now classical that
an ``intermediate'' equation between \eqref{A3} and \eqref{A4} can be derived from \eqref{A3} by working in a moving frame
$x=ct$, $c\in \R$. Indeed, letting
\be
\label{A5}
z(x,t)=y(x-ct,t)
\ee
we readily see that \eqref{A3} is transformed into the following KdV-BBM equation
\be
\label{A6}
z_t+(c+1)z_x-cz_{xxx}-z_{txx}+zz_x=0.
\ee
It is then reasonable to expect the control properties of \eqref{A6} to be better than those of \eqref{A3}, thanks to the KdV term $-cz_{xxx}$ in \eqref{A6}.
In \cite{RZpreprint}, it was proved that the equation \eqref{A6} with a forcing term supported in (any given) subdomain is locally exactly controllable in $H^1(\T)$
provided that $T>(2\pi)/c$. Going back to the original variables, it means that the equation
\be
\label{A7}
y_t + y_x -y_{txx} + yy_x = b(x+ct) h(x,t)
\ee
with a moving distributed control is exactly controllable in $H^1(\T )$ in (sufficiently) large time. Actually, this control time has to be chosen in such a way that
the support of the control, which is moving at the constant velocity $c$,  can visit all the domain $\T$.

The concept of  moving point control was introduced by J. L. Lions in \cite{Lions} for the wave equation. One important motivation for this kind of control is that the exact controllability
of the wave equation with a pointwise control and Dirichlet boundary conditions fails if the point is a zero of some eigenfunction of the Dirichlet Laplacian, while it holds
 when the point is moving under some (much more stable) conditions easy to check
(see e.g. \cite{Castro}).
The controllability of the wave equation (resp. of the heat equation) with a moving point control was investigated in \cite{Lions,Khapalov1,Castro} (resp. in \cite{Khapalov2,CZ2}).
See also \cite{Zhang} for Maxwell's equations.

As the bad control properties of \eqref{A1} come from the BBM term $-\varepsilon y_{txx}$, it is natural to ask whether better control properties for \eqref{A1}
could be obtained by using a moving control, as for the BBM equation in \cite{RZpreprint}. The aim of this paper is to investigate that issue.

Throughout the paper, we will take $\varepsilon =1$ for the sake of simplicity. All the results can be extended without difficulty to any $\varepsilon >0$.
Let $y$ solve
\be
\label{A8}
y_{tt} - y_{xx} -y_{txx} = b(x+ct) h(x,t).
\ee
Then $v(x,t)=y(x-ct,t)$ fulfills
\be
\label{A9}
v_{tt} + (c^2-1) v_{xx} +2c v_{xt} -v_{txx} -c v_{xxx} = b(x) \tilde h (x,t)
\ee
where $\tilde h( x,t) =h(x-ct, t)$.
Furthermore the new initial condition read
\be
\label{condini}
v(x,0)=y(x,0),\quad v_t (x,0)= - c\, y_x (x,0)+y_t(x,0).
\ee

As for the KdV-BBM equation, the appearance of a KdV term (namely $-cv_{xxx}$ in \eqref{A9}) results in much better control properties.
We shall see that
\begin{enumerate}
\item[(i)] there is no accumulation point in the spectrum of the free evolution equation ($\tilde h=0$ in \eqref{A9});
\item[(ii)]  the spectrum splits into one part of ``parabolic'' type, and another part of ``hyperbolic'' type.
\end{enumerate}
It follows that  one can expect at most a null controllability result in large time. We will see that this is indeed the case.
Throughout the paper, we assume that $c=-1$ for the sake of simplicity.
Let us now state the main results of the paper.
We shall denote by $(y_0,\xi _0)$ an initial condition (taken in some appropriate space) decomposed in Fourier series as
\be
y_0(x)=\sum_{k\in \Z} c_k e^{ikx}, \quad \xi _0(x) =\sum_{k\in \Z} d_k e^{ikx}.
\label{Fourier}
\ee
We shall consider several control problems. The first one reads
\ba
&&y_{tt}-y_{xx}-y_{txx} = b(x-t)h(t), \qquad x\in \T, t>0,\label{cont1}\\
&&y(x,0)=y_0(x),\ y_t(x,0)=\xi _0 (x),\qquad x\in \T \label{cont1bis}
\ea
where $h$ is the scalar control.
\begin{theorem}
\label{thm1}
Let $b\in L^2(\T )$ be such that
\[
\beta _k=\int_{\T} b(x) e^{-ikx} dx\ne 0\quad \text{ for }  k\ne 0,\quad \beta _0 =\int_{\T } b(x)\, dx =0.
\]
For any time $T>2\pi$ and any $(y_0,\xi _0)\in L^2(\T )^2$ decomposed as in \eqref{Fourier}, if
\be
\label{condthm1}
\sum_{k\ne 0} |\beta _k|^{-1} (|k|^6 |c_k| + |k|^4 |d_k| ) <\infty \quad \text{ and } c_0  = d_0 =0,
\ee
then there exists a control $h\in L^2(0,T)$ such that the solution of \eqref{cont1}-\eqref{cont1bis} satisfies
$y(.,T)=y_t(.,T) = 0$.
\end{theorem}
By Lemma~\ref{lem:IrrationalInterval}  (see below) there exist simple functions~$b$ such that $|\beta_k|$ decreases like~$1/|k|^3$, so that
~\eqref{condthm1} holds for $(y_0,\xi _0)\in H^{s+2}(\T )\times H^s(\T )$ with $s > 15/2$.

The second problem we consider is
\ba
&&y_{tt}-y_{xx}-y_{txx} = b(x-t)h(x,t), \qquad x\in \T, t>0,\label{cont2}\\
&&y(x,0)=y_0(x),\ y_t(x,0)=\xi _0 (x),\qquad x\in \T , \label{cont2bis}
\ea
where the control function $h$ is here allowed to depend also on $x$. For that internal controllability problem,
the following result will be established.
\begin{theorem}
\label{thm2}
Let $b={\mathbf 1}_{\omega}$ with $\omega$ a nonempty open subset of $\T$.
Then for any time $T>2\pi$ and any $(y_0,\xi _0)\in H^{s+2}(\T )\times H^s(\T ) $ with $s > 15/2$
there exists a control $h\in L^2(\T\times(0,T))$ such that the solution of \eqref{cont2}-\eqref{cont2bis} satisfies
$y(.,T)=y_t(.,T) = 0$.
\end{theorem}

We now turn our attention to some internal controls acting on a single moving point.
The first problem we consider reads
\ba
&&y_{tt}-y_{xx}-y_{txx} = h(t)\delta _{t}, \qquad x\in \T,\  t>0,\label{cont3}\\
&&y(x,0)=y_0(x),\ y_t(x,0)=\xi _0 (x),\qquad x\in \T,\label{cont3bis}
\ea
where $\delta _{x_0}$ represents the Dirac measure at $x=x_0$. We can as well replace
$\delta_{t}$ by $\frac{d \delta_{t}}{dx}$ in \eqref{cont3}, which yields another control problem:
\ba
&&y_{tt}-y_{xx}-y_{txx} = h(t)\frac{d\delta _t}{dx}, \qquad x\in \T,\  t>0,\label{cont4}\\
&&y(x,0)=y_0(x),\ y_t(x,0)=\xi _0 (x),\qquad x\in \T.\label{cont4bis}
\ea

Then we will obtain the following results.
\begin{theorem}
\label{thm3}
For any time $T>2\pi$ and any $(y_0,\xi _0)\in H^{s+2}(\T )\times H^s(\T ) $ with $s>9/2$, 
there exists a control $h\in L^2(0,T)$ such that the solution of \eqref{cont3}-\eqref{cont3bis} satisfies
$y(T,.)-[y(T,.)]=y_t(T,.) = 0$, where $[f]=(2\pi )^{-1}\int_0^{2\pi} f(x)dx$ is the mean value of $f$.
\end{theorem}

\begin{theorem}
\label{thm4}
For any time $T>2\pi$ and any $(y_0,\xi _0)\in H^{s+2}(\T )\times H^s(\T ) $ with $s>7/2$ and such that $\int_{\T} y_0(x)dx  = \int_{\T } \xi _0(x) dx =0$,
there exists a control $h\in L^2(0,T)$ such that the solution of \eqref{cont4}-\eqref{cont4bis} satisfies
$y(T,.)=y_t(T,.) = 0$.
\end{theorem}

The paper is organized as follows. Section \ref{Sec:Section2} is devoted to the proofs of the above theorems: in subsection~\ref{ssec:spectral} we investigate the
wellposedness and the spectrum of~\eqref{A9} for $c=-1$; in subsection~\ref{ssec:moments} the null controllability of~\eqref{cont1}-\eqref{cont1bis},
\eqref{cont3}-\eqref{cont3bis} and \eqref{cont4}-\eqref{cont4bis} are formulated as moment problems; Theorem~\ref{thm1} is proved in subsection~\ref{ssec:ProofThm1}
thanks to a suitable biorthogonal family which is  shown to exist in Proposition~\ref{prop1}; Theorem~\ref{thm2} is deduced from
 Theorem~\ref{thm1} in subsection~\ref{ssec:ProofThm2}; finally, the proofs of Theorems ~\ref{thm3} and~\ref{thm4}, that are almost identical to the proof of Theorem~\ref{thm1},  are sketched
in subsection~ \ref{ssec:ProofThm34}.
The rather long proof of Proposition~\ref{prop1} is postponed to Section \ref{Sec:Section3}. It combines different results of complex analysis about entire functions of exponential type, sine-type functions,  atomization of measures, and Paley-Wiener theorem.

\section{Proof of the main results}
\label{Sec:Section2}

\subsection{Spectral decomposition} \label{ssec:spectral}
The free evolution equation associated with \eqref{A9} reads
\be
\label{A10}
v_{tt}-2v_{xt} -v_{txx} +v_{xxx}=0.
\ee
Let $v$ be as in \eqref{A10}, and let $w=v_t$. Then \eqref{A10} may be written as
\be
\label{A11}
\left(
\begin{array}{c}
v\\w
\end{array}
\right) _t
=A
\left(
\begin{array}{c}
v\\w
\end{array}
\right)
:=
\left(
\begin{array}{c}
w\\
2w_x + w_{xx} -v_{xxx}
\end{array}
 \right) .
\ee
The eigenvalues of $A$ are obtained by solving the system
\be
\label{A12}
\left\{
\begin{array}{l}
w=\lambda v,\\
2\lambda v_x + \lambda v_{xx} -v_{xxx} =\lambda ^2 v.
\end{array}
\right.
\ee
Expanding $v$ as a Fourier series $v=\sum_{k\in \Z} v_k e^{ikx}$, we see that \eqref{A12} is satisfied
provided that for each $k\in \Z$
\be
\label{A14}
(\lambda ^2 +(k^2-2ik)\lambda -ik^3)v_k=0.
\ee
For $v_k\ne 0$, the only solution of \eqref{A14} reads
\be
\label{ZZ1}
\lambda =\lambda _k^{\pm } = \frac{-(k^2-2ik) \pm \sqrt{k^4-4k^2}}{2}\cdot
\ee
Note that
$$
\lambda _0^\pm =0,\ \lambda _2^\pm =-2+2i,\ \lambda _{-2} ^\pm =-2-2i
$$
while
$$\lambda _k^+\ne \lambda _l^- \ \text{ for }\  k,l\in \Z \setminus \{ 0,\pm 2 \} \quad \text{with } k\ne l.$$
For $|k| \ge 3$, $\lambda _k ^{\pm} = \frac{-k^2 \pm k^2 (1- 2 k^{-2} +O(k^{-4}) ) }{2} +ik$.
Hence
\begin{eqnarray}
\lambda _k^+ &=& -1 + ik + O(k^{-2}) \quad \text{ as } |k|\to \infty ,\label{A16}\\
\lambda _k^-  &=& -k^2 +1 +ik +O(k^{-2}) \quad \text{ as } |k|\to \infty \label{A17}.
\end{eqnarray}
The spectrum $\Lambda =\{ \lambda_k^\pm ;\  k\in \Z \}$ may be split  into $\Lambda = \Lambda ^+ \cup \Lambda ^- \cup \Lambda _2$ where\
\begin{eqnarray*}
\Lambda ^+&=& \left\{ \lambda _k^+;\ k\in \Z \setminus \{ 0 ,\pm 2 \}  \right\},\\
\Lambda ^-&=&  \left\{ \lambda _k^-;\ k\in \Z \setminus \{ 0 ,\pm 2 \}  \right\},\\
\Lambda _2 &=& \{ 0, -2  \pm 2i \}
\end{eqnarray*}
denote the {\em hyperbolic} part, the {\em parabolic} part, and the set of double eigenvalues, respectively. It is displayed on Figure~\ref{fig:spectrum}.
(See also \cite{LZ} for a system  whose spectrum may also be decomposed into a hyperbolic  part and a parabolic part.)
\begin{figure*}[!t]
\centerline{\includegraphics[width=150mm]{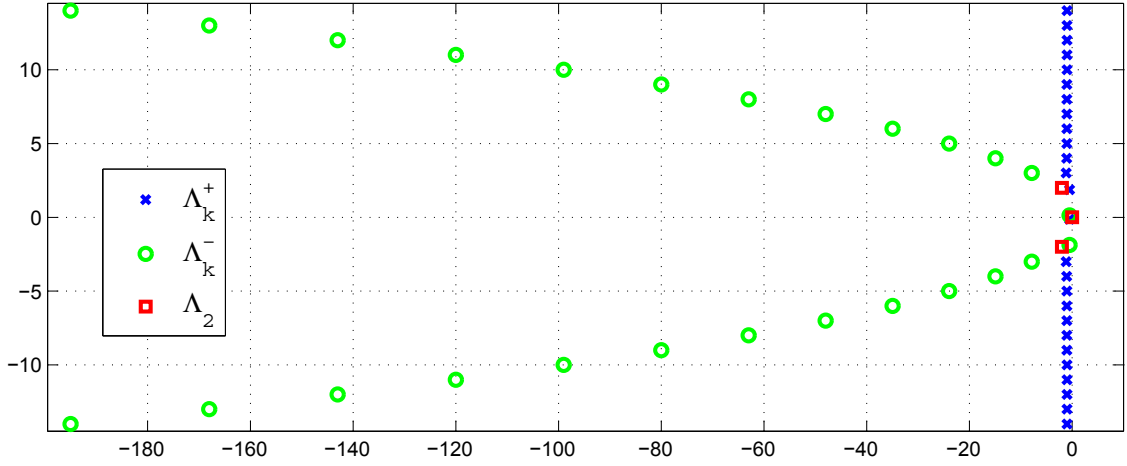}}
\caption{Spectrum of~\eqref{A10} splits into a hyperbolic part ($\Lambda^+_k$ in blue), a parabolic part ($\Lambda^-_k$ in green) and a finite dimensional part ($\Lambda_2$ in red).}
\label{fig:spectrum}
\end{figure*}

An eigenvector associated with the eigenvalue $\lambda _k^\pm$, $k\in \Z$, is
$\left(
\begin{array}{c}
e^{ikx}\\
\lambda_k^\pm e^{ikx}
\end{array}  \right)
$,
and the corresponding  exponential solution of \eqref{A10} reads
$$v_k^\pm (x,t) = e^{\lambda _k^\pm t} e^{ikx}.$$
For $k\in \{ 0,\pm 2\}$, we denote $\lambda _k=\lambda _k^+=\lambda _k^-$, $v_k(x,t)=e^{\lambda _k t } e^{ikx}$, and introduce
$$
\tilde v_k(x,t) : =t e^{\lambda _k t}e^{ikx}.
$$
Then we easily check that $\tilde v_k$ solves \eqref{A10} and
$$
\left( \begin{array}{c}
\tilde v_k\\
\tilde v_{kt}
 \end{array} \right)  (x,0) =
 \left(
 \begin{array}{c}
 0\\
 e^{ikx}
 \end{array}
 \right) .
$$
Any solution of \eqref{A10} may be expressed in terms of the $v_k^\pm$'s, the $v_k$'s, and  the $\tilde v_k$'s.
Introduce first the Hilbert space $\cH = H^1(\T ) \times L^2(\T )$ endowed with the scalar product
$$
\langle (v_1,w_1),(v_2,w_2) \rangle _{\cH} =\int_{\T}
[(v_1\overline{v_2}+v_1'\overline{v_2'}) +w_1\overline{w_2}]dx.
$$
Pick any
\be
\label{A20}
\left(
\begin{array}{c}
v_0\\
w_0
\end{array}
\right)
=\left(
\begin{array}{l}
\sum_{k\in\Z} c_k e^{ikx}\\
\sum_{k\in\Z}  d_k e^{ikx}
\end{array}
\right) \in \cH .
\ee
For $k\in \Z \setminus \{ 0,\pm 2 \} $, we write
\be
\label{A21}
\left(
\begin{array}{c}
c_ke^{ikx}\\
d_ke^{ikx}
\end{array}
 \right)
 =a_k^+
 \left(
\begin{array}{c}
e^{ikx}\\
\lambda _k^+ e^{ikx}
\end{array}
 \right)
+  a_k^-
\left(
\begin{array}{c}
e^{ikx}\\
\lambda_k^-  e^{ikx}
\end{array}
 \right)
\ee
with
\ba
a_k^+ &=& \frac{d_k-\lambda_k^- c_k}{\lambda_k^+ - \lambda_k^-}, \label{A22}\\
a_k^- &=& \frac{d_k-\lambda_k^+ c_k}{\lambda_k^- - \lambda_k^+}\cdot\label{A23}
\ea
For $k\in \{ 0,\pm 2\}$, we write
\be
\label{A24}
\left(
\begin{array}{c}
c_k e^{ikx}\\
d_k e^{ikx}
\end{array}
 \right)
 =a_k
 \left(
\begin{array}{c}
e^{ikx}\\
\lambda _k e^{ikx}
\end{array}
 \right)
+  \tilde a_k
\left(
\begin{array}{c}
0\\
e^{ikx}
\end{array}
 \right)
\ee
with
\be
a_k=c_k,\quad  \tilde a_k = d_k - \lambda _k c_k.
\label{XX1}
\ee
It follows that the solution $(v, w)$ of
\be\label{A26}
\left( \begin{array}{c}
v\\w
\end{array}
\right) _t =
A \left( \begin{array}{c}
v\\w
\end{array}  \right),\quad
\left( \begin{array}{c}
v\\w
\end{array}
\right) (0) =
\left( \begin{array}{c}
v_0\\w_0
\end{array}  \right)
\ee
may be decomposed as
\ba
\left( \begin{array}{c}
v(x,t)\\w(x,t)
\end{array}
\right)  =\sum_{k\in \Z \setminus \{0,\pm 2 \} }
\{ a_k^+ e^{\lambda _k^+ t}
\left( \begin{array}{c}
e^{ikx}\\ \lambda _k^+ e^{ikx}
\end{array}  \right) +
a_k^-  e^{\lambda _k^- t}
\left( \begin{array}{c}
e^{ikx}\\ \lambda _k^- e^{ikx}
\end{array}  \right)
\}\nonumber \\
+ \sum_{k\in \{0,\pm 2 \} }
\{
a_k e^{\lambda _k t}
\left( \begin{array}{c}
e^{ikx}\\ \lambda _k  e^{ikx}
\end{array}  \right)
+ \tilde a_k e^{\lambda _k t}
\left( \begin{array}{c}
t e^{ikx}\\
(1+\lambda _k t) e^{ikx}
\end{array}  \right)\}.
\label{A27}
\ea
\begin{proposition} \label{prop:cauchy}
Assume that $(v_0,w_0) \in H^{s+1}(\T ) \times H^s ( \T )$ for some $s\ge 0$. Then the solution
$(v,w)$ of \eqref{A26} satisfies $(v,w)\in C([0,+\infty ) ; H^{s+1}(\T )\times H^s(\T ))$.
\end{proposition}
\begin{proof}
Assume first that $(v_0,w_0)\in C^\infty (\T )\times C^\infty (\T )$. Decompose
$(v_0,w_0)$ as in \eqref{A20}, and let $a_k^\pm$ for $k\in \Z\setminus \{ 0, \pm 2\} $, and $a_k, \tilde a _k$ for $k\in \{ 0, \pm 2\}$, be as in \eqref{A22}-\eqref{A23}
and \eqref{XX1}, respectively. Then, from the classical Fourier definition of
Sobolev spaces, we have that
\begin{eqnarray*}
||(v_0,w_0)||_{ H^{s+1}(\T ) \times H^s(\T )} &\sim& \left(  \sum_{k\in \Z} (|k|^ 2 + 1)^{s} \big( (|k|^2+1) |c_k|^2 + |d_k|^2 \big) \right) ^{\frac{1}{2}} \\
&\sim& \left( \sum_{k\in \Z \setminus \{ 0, \pm 2\} } |k|^{2s} (k^2 |a_k^+|^2 +k^4 |a_k^-| ^2)  + \sum_{k\in \{0, \pm 2\} } (|a_k|^2+|\tilde a_k|^2 ) \right)^{\frac{1}{2}} .
\end{eqnarray*}
For the last equivalence of norms, we used \eqref{A16}-\eqref{A17} and \eqref{A21}-\eqref{A23}. Since
$$
|e^{\lambda _k^+ t}| + |e^{\lambda _k^- t } | \le C\quad \text{ for } |k|>2,\ t\ge 0,
$$
we infer that
$$
\sum_{k\in \Z \setminus \{ 0, \pm 2\} } |k|^{2s} \left(  k^2 |a_k^+  e^{\lambda _k^+ t} | ^2 + k^4 | a_k^-  e^{\lambda _k^- t} |^2 \right)
\le C \sum_{ k\in \Z \setminus \{ 0, \pm 2\} } |k|^{2s} \left(  k^2 |a_k^+ | ^2  + k^4 | a_k^- |^2 \right)   <\infty,
$$
hence
\be
||(v,w)||_{L^\infty (\R ^+,\  H^{s+1}(\T )\times H^s(\T ))} \le C||(v_0,w_0)||_{H^{s+1}(\T ) \times H^s (\T )} \cdot
\label{A30}
\ee
The result follows from \eqref{A27} and \eqref{A30} by  a density argument.
\end{proof}

\subsection{Reduction to moment problems} \label{ssec:moments}
\subsubsection{Internal control}
We investigate the following control problem
\be
v_{tt} - 2v_{xt} -v_{txx} +v_{xxx} =b(x)h(t),
\label{B1}
\ee
where $b\in L^2(\T )$, $\hbox{supp }b\subset \omega \subset \T$ and  $h\in L^2(0,T)$.
The adjoint equation to \eqref{B1} reads
\be
\label{B4}
\varphi_{tt}-2\varphi_{xt} +\varphi_{txx} -\varphi_{xxx} = 0.
\ee
Note that $\varphi(x,t) = v(2\pi-x,T-t)$ is a solution of \eqref{B4} if $v$ is a solution of
\eqref{B1} for $h\equiv 0$.  Pick any (smooth enough) solutions $v$ of \eqref{B1} and $\varphi$ of \eqref{B4}, respectively.
Multiplying each term in \eqref{B1}  by $\overline{\varphi}$ and integrating by parts, we obtain
\be
\int_{\T } [v_t \overline{\varphi} + v ( -\overline{\varphi} _t + 2 \overline{\varphi}_x -\overline{\varphi}_{xx} ) ]  \bigg\vert_0^T dx =
\int_0^T\!\!\!\int_{\T} hb\overline{\varphi}\, dxdt.
\label{B5}
\ee
Pick first $\varphi (x,t)=e^{\lambda _{-k}^\pm  (T-t)} e^{ikx}
=e^{\overline{\lambda _k^\pm}  (T-t)} e^{ikx}$ for $k\in\Z$.
Then  \eqref{B5} may be written
\ba
\lan v_t(T), e^{ikx}\ran  +( \lambda _k^\pm -2ik +k^2 )\lan v(T), e^{ikx} \ran
-e^{\lambda _k ^\pm T } \gamma _k ^\pm \nonumber \\
= \int_0^T h(t) e^{\lambda _k^\pm (T-t) }dt ~ \int_{\T}  b(x) e^{-ikx}dx,
\label{B6}
\ea
where $\langle .,.  \rangle$ stands for the duality pairing $\langle.,. \rangle_{{\mathcal D}'(\T ), {\mathcal D}(\T ) } $, and
\[
\gamma _k^\pm =\lan v_t(0), e^{ikx} \ran + (\lambda _k^\pm -2ik + k^2 ) \lan v(0), e^{ikx} \ran .
\]
If we now pick $\varphi (x,t)=(T-t)e^{\overline{\lambda _k}(T-t)} e^{ikx}$ for $k\in \{ 0, \pm 2\}$, then
\eqref{B5} yields
\ba
\lan v (T), e^{ikx}\ran  -\left\{ T e^{\lambda _k T} \lan v_t(0),  e^{ikx}  \ran
+ [1+T (\lambda _k -2ik + k ^2) ] e^{\lambda _k T} \lan v(0), e^{ikx} \ran \right\}  \nonumber \\
= \int_0^T (T-t) h(t) e^{\lambda _k(T-t) }dt ~ \int_{\T}  b(x) e^{-ikx}dx \quad k\in \{0, \pm 2 \} .
\label{B7}
\ea
Set $\beta _k= \int_{\T} b(x) e^{-ikx}dx$ for $k\in \Z$. The control problem can be reduced to a moment problem.
Assume that there exists some function $h\in L^2(0,T)$ such that
\ba
&&\beta _k \int_0^T e^{\lambda _k^\pm (T-t) }h(t) dt = - e^{\lambda _k^\pm T} \gamma _k^\pm \quad \forall k\in \Z ,\label{B8}\\
&&\beta_k  \int_0^T (T-t) e^{\lambda _k (T-t)} h(t) dt  \nonumber \\
&&\quad = -T   e^{\lambda _k T}  \lan v_t(0) , e^{ikx} \ran
- [1+T(\lambda _k -2ik + k^2) ] e^{\lambda _k T } \lan v(0), e^{ikx}\ran \quad \forall k\in \{ 0, \pm 2  \}.\qquad  \label{B8bis}
\ea
Then it follows from \eqref{B6}-\eqref{B8bis} that
\ba
\lan v_t(T), e^{ikx} \ran + (\lambda _k^\pm -2ik + k^2 ) \lan  v(T) ,  e^{ikx} \ran &=&0 \quad \forall k\in \Z , \label{B10} \\
\lan v(T), e^{ikx} \ran &=&0\quad \forall k\in \{0, \pm 2 \}. \label{B11}
\ea
Since $\lambda _k^+ \ne \lambda _k^- $ for $k\in \Z \setminus \{ 0, \pm 2  \}$, this yields
\be
v(T)=v_t(T)=0. \label{B12}
\ee
\subsubsection{Point control}
Let us consider first the control problem
\be
v_{tt} - 2 v_{xt} -v_{txx} + v_{xxx} = h(t) \frac{d\delta _0}{dx}\cdot
\label{B13}
\ee
Then the right hand side of \eqref{B5} is changed into $\int_0^T h(t) \lan \frac{d\delta _0}{dx},\varphi  \ran dt$. For
$\varphi (x,t) = e^{\overline{\lambda _k^\pm }(T-t)} e^{ikx}$, we have
\[
\lan \frac{d\delta _0}{dx},\varphi  \ran =-\lan \delta _0, \frac{\partial \varphi}{\partial x}\ran =ik \, e^{\lambda _k^\pm (T-t)}
\]
hence the right hand sides of \eqref{B6} and \eqref{B7} are changed into $\int_0^T (ik) e^{\lambda _k^\pm (T-t)} h(t) dt$ and
$\int_0^T (ik) (T-t) e^{\lambda _k (T-t)} h(t) dt$, respectively.
Let $\beta _k = ik$ for $k\in \Z$. Note that $\beta _0=0$ and that \eqref{B6}-\eqref{B7} for $k=0$ read
\ba
&&\lan v_t(T), 1 \ran  - \lan  v_t(0) , 1 \ran =0, \label{B15}\\
&&\lan v(T), 1 \ran - T \lan v_t(0), 1\ran - \lan  v(0), 1\ran =0. \label{B16}
\ea
Thus, the mean values of $v$ and $v_t$ cannot be controlled. Let us formulate the moment problem to be solved.
Assume that
\be
\lan v(0), 1  \ran = \lan v_t(0), 1 \ran =0,
\label{B17}
\ee
and that there exists some $h\in L^2(0,T)$ such that
\ba
&&ik\int_0^T e^{\lambda _k ^\pm (T-t)}h(t) dt = -
e^{\lambda _k^\pm T} \gamma _k ^\pm \quad \forall k\in \Z \setminus \{0\} ,
\label{B18}\\
&& ik\int_0^T (T-t) e^{\lambda _k  (T-t)}h(t) dt \nonumber \\
&&\quad =
 - T e^{\lambda _k T}  \lan v_t(0) , e^{ikx} \ran
 -  [1+T(\lambda _k -2ik+ k^2) ]e^{\lambda _kT} \lan v(0) , e^{ikx}\ran \quad \forall k\in \{\pm 2\} .
\label{B19}
\ea
Then we infer from \eqref{B6}-\eqref{B7} (with the new r.h.s.) and \eqref{B15}-\eqref{B19} that
\[
v(T)=v_t(T)=0.
\]
Finally, let us consider the control problem
\be
v_{tt} - 2 v_{xt} -v_{txx} + v_{xxx} = h(t) \delta _0\cdot
\label{B130}
\ee
Then the computations above are valid with the new values of $\beta _k$ given by
\[
\beta _k = \langle \delta _0, e^{ikx} \rangle =1,\qquad k\in \Z.
\]
It will be clear from the proof of Theorem \ref{thm1}
that $\langle v_t(T), 1 \rangle$ can be controlled, while $\langle v(T), 1\rangle$ cannot.
To establish Theorem \ref{thm3}, we shall have to find a control function $h\in L^2(0,T)$ such that
\ba
&&\int_0^T e^{\lambda _k ^\pm (T-t)}h(t) dt = -
e^{\lambda _k^\pm T} \gamma _k ^\pm \quad \forall k\in \Z  ,
\label{B180}\\
&& \int_0^T (T-t) e^{\lambda _k  (T-t)}h(t) dt \nonumber \\
&&\quad =
 - T e^{\lambda _k T}  \lan v_t(0) , e^{ikx} \ran
 -  [1+T(\lambda _k -2ik+ k^2) ]e^{\lambda _kT} \lan v(0) , e^{ikx}\ran \quad \forall k\in \{\pm 2\} .
\label{B190}
\ea
\subsection{A Biorthogonal family}\label{ssec:biorthogonal}

To solve the moments problems in the previous section, we need to
construct a biorthogonal family to the functions $e^{\lambda _k^\pm t}$, $k\in \Z$, and
$t\, e^{\lambda _k t}$, $k\in \{ \pm 2\}$. More precisely, we shall prove the following
\begin{proposition}
\label{prop1}
There exists a family $\{ \psi_k^\pm \} _{k\in\Z \setminus \{0,\pm 2\} } \cup \{ \psi _k \} _{k\in \{ 0,\pm 2 \} } \cup \{ \tilde \psi _k\}_{k\in \{ \pm 2\} }$
of functions in $L^2(-T/2,T/2)$  such that
\ba
&&\int_{-T/2}^{T/2} \psi _k^\pm (t) e^{\lambda _l ^\pm t }dt = \delta_k^l\delta _+^-\qquad k,l\in \Z \setminus \{0,\pm 2\}, \label{E1}\\
&&\int_{-T/2}^{T/2}\psi_k^\pm (t) e^{\lambda _l t}dt= \int_{-T/2}^{T/2} \psi_k ^\pm (t) t e^{\lambda _p t}dt =0 \qquad  k\in \Z \setminus \{0, \pm 2\},\ l\in \{ 0,\pm 2\},
\   p\in \{ \pm 2 \} , \qquad \qquad \label{E2} \\
&&\int_{-T/2}^{T/2} \psi _l (t) e^{\lambda _k^\pm t}dt = \int_{-T/2}^{T/2} \tilde \psi _p (t) e^{\lambda _k ^\pm t } dt =0
\qquad l\in \{ 0,\pm 2\},\  k\in \Z \setminus \{0, \pm 2\},\  p\in \{ \pm 2 \},\qquad   \label{E3} \\
&&\int_{-T/2}^{T/2} \psi _l (t) e^{\lambda _k t} dt = \delta _l^k,\quad
\int_{-T/2}^{T/2} \psi _l (t) t e^{\lambda _p t} dt=0\qquad
l,k \in \{ 0,\pm 2 \} , \   p\in \{ \pm  2\}, \label{E4}\\
&& \int_{-T/2}^{T/2} \tilde \psi _p (t) e^{\lambda _k t}dt =0, \quad
\int_{-T/2}^{T/2} \tilde \psi _p(t) t e^{\lambda _q t } dt
=\delta _p^q \qquad p,q \in \{\pm 2\}, \  k\in \{ 0,\pm 2\},  \label{E5} \\
&&||\psi _k^+||_{L^2(-T/2,T/2)} \le C |k|^4 \qquad\qquad\qquad \qquad \qquad   k\in \Z \setminus \{0, \pm 2\},\qquad \label{E6}\\
&&||\psi _k^-||_{L^2(-T/2,T/2)} \le C |k| ^2 e^{-\frac{T}{2} k^2 +2\sqrt{2} \pi |k|} \qquad \qquad k\in\Z \setminus \{ 0, \pm 2\}, \label{E7}
\ea
where $C$ denotes some positive constant.
\end{proposition}
In Proposition \ref{prop1}, $\delta _k^l$ and $\delta _+^-$ denote Kronecker symbols ($\delta _k^l=1$ if $k=l$, 0 otherwise, while
$\delta _+^-=1$ if we have the same signs in the l.h.s of \eqref{E1}, 0 otherwise). The proof of Proposition \ref{prop1} is postponed to Section~\ref{Sec:Section3}.
We assume Proposition \ref{prop1} true for the time being and proceed to the proofs of the main results of the paper.

 \subsection{Proof of Theorem \ref{thm1}}\label{ssec:ProofThm1}

 Pick any pair  $(y_0,\xi _0)\in L^2(\T )^2$ fulfilling \eqref{condthm1}. From \eqref{condini} with $c=-1$,  we have that $v(0)=y_0$, $v_t(0)=\frac{dy_0}{dx} + \xi _0$, so that
 \begin{eqnarray*}
 \gamma _k ^\pm  &=& \langle \frac{d y_0}{dx} + \xi  _ 0, e^{ikx} \rangle  + (\lambda _k^\pm -2ik + k^2 ) \langle y_0 , e^{ikx} \rangle ,\\
 &=& \langle \xi  _ 0, e^{ikx} \rangle  + (\lambda _k^\pm -ik + k^2 ) \langle y_0 , e^{ikx} \rangle ,
 \qquad k\in \Z .
 \end{eqnarray*}
 Let
 \[
 \gamma _k = \gamma _k^\pm \quad \text{ for } k\in \{0,\pm 2\}.
 \]
 The result will be proved if we can construct a control function $h\in L^2(0,T)$ fulfilling \eqref{B8}-\eqref{B8bis}.
 Let us introduce the numbers
 \begin{eqnarray*}
 \alpha _k^\pm &=&  - \beta _k^{-1} e^{\lambda _k^\pm \frac{T}{2}}\gamma _k^\pm , \qquad k\in \Z \setminus \{ 0,\pm 2\},\\
 \alpha _k &=&  - \beta _k^{-1} e^{\lambda _k \frac{T}{2}}\gamma _k , \qquad k\in \{ \pm 2\},\\
 \tilde \alpha _k &=&  - \beta _k^{-1} \left( \frac{T}{2}e^{\lambda _k \frac{T}{2}}   \gamma _k  + e^{\lambda _k \frac{T}{2}} \langle y_0, e^{ikx}\rangle
  \right) ,\qquad k\in \{ \pm 2 \},
 \end{eqnarray*}
and
\[
\psi (t) = \sum_{k\in \Z \setminus\{ 0, \pm 2 \} } \alpha _k ^+ \psi _k ^+ (t) + \sum_{k\in \Z \setminus \{ 0, \pm 2\} } \alpha _k ^- \psi _k ^- (t)  
+ \sum_{k \in \{\pm 2 \} }  [ \alpha _k  \psi _k (t) + \tilde\alpha _k \tilde \psi _k (t) ].
\]
Finally let $h(t)= \psi (\frac{T}{2} -t)$. Note that $h\in L^2(0,T)$ with
\begin{eqnarray*}
||h||_{L^2(0,T)}&=&||\psi ||_{L^2(-\frac{T}{2},\frac{T}{2})}\\
&\le&C\left( \sum_{k\in \{ \pm 2\} } (|d_k|+|c_k|)  + \sum_{k\in\Z \setminus \{ 0, \pm 2\}} |\beta _k|^{-1} (|d_k| + |k|^2 |c_k|)|k|^4 \right.\\
&&\quad\left.  +   \sum_{k\in \Z \setminus \{ 0,\pm 2\} } |\beta _k |^{-1} (|d_k|+|c_k|) |k|^2 e^{-T|k|^2+ 2\sqrt{2}\pi |k|}\right) \\
&<&\infty,
\end{eqnarray*}
by \eqref{condthm1}.
Then it follows from \eqref{condthm1} and \eqref{E1}-\eqref{E5} that for $k\in \Z \setminus \{ 0, \pm 2\}$
\[
\beta _k \int_0^T e^{\lambda _k^\pm (T-t)}h(t)dt = \beta _k e^{\lambda _k^\pm  \frac{T}{2} } \int_{-T/2}^{T/2}e^{\lambda _k ^\pm \tau }\psi (\tau ) \, d\tau = \beta _k e^{\lambda _k ^\pm
\frac{T}{2}} \alpha _k^\pm =  - e^{\lambda _k^\pm T } \gamma _k^\pm.
\]
and also that
\begin{eqnarray*}
\beta _k \int _0^T e^{\lambda _k (T-t)} h(t) dt &=& - e^{\lambda _k T} \gamma _k  \qquad \text{ for } k\in \{ 0,\pm 2\}  ,\\
\beta _k \int_0^T (T-t) e^{\lambda _k (T-t)} h(t)dt &=&-  Te^{\lambda _k T} \gamma _k - e^{\lambda _k T} \langle y_0, e^{ikx}\rangle \qquad
\text{ for } k\in \{0, \pm 2\},
\end{eqnarray*}
as desired. \qed

\subsection{Proof of Theorem \ref{thm2}}\label{ssec:ProofThm2}
Set $\epsilon=(T-2\pi)/2$, $v(x,t)=y(x+t,t)$ and $\xi (x,t)=y_t (x,t)$.
We first steer to $0$ the components of $v$ and $v_t$ along the mode associated to the double
eigenvalue $\lambda_0=0$. Denote $\gamma(t)=\int_{\T }  v(x,t)~dx$ and $\eta(t)=\int_{\T }  v_t(x,t)~dx$.
According to~\eqref{Fourier}, $\gamma(0)= 2\pi c_0$,  $\eta(0)=2\pi d_0$ and
$$
\frac{d\gamma}{dt} = \eta, \quad \frac{d\eta}{dt}= \int_{\omega} \tilde h(x,t) dx.
$$
Take a $C^\infty$ scalar function $\varpi(t)$ on $[0,\epsilon]$ with $\varpi(0)=1$ and $\varpi(\epsilon)=0$  and such
that the support of $d\varpi/dt$ lies inside $[0,\epsilon]$. Consider another $C^\infty$ function of $x$, $\bar b(x)$ with support inside $\omega$ and such that
$\int_\omega \bar  b(x)\, dx=1$. Then the  $C^\infty$ control
$$
\tilde h(x,t)=\bar b(x) \bar h(t)~\text{ with }~\bar h(t)=\frac{d^2}{dt^2} \left( (c_0+d_0 t) \varpi(t)\right)
$$
steers $(\gamma,\eta)$ from $(c_0,d_0)$ at time $t=0$ to $(0,0)$ at time $t=\epsilon$. Its support lies inside $[0,\epsilon]$.
Since $\gamma(\epsilon)=\int_{\T }  y(x,\epsilon)\, dx$ and $\eta(\epsilon)=\int_{\T }  \xi(x,\epsilon)\, dx$, we can assume that $c_0=d_0=0$
up to a time shift of $\epsilon$.

Since $\omega$ is open and nonempty, it contains a small interval $[a,a+2\sigma\pi]$ where $\sigma >0$ is a  {\em quadratic irrational};
 i.e.,  an irrational number which is  a root  of a quadratic equation with integral coefficients. Set for $t\in[\epsilon,T]$
$$
 h(x,t)= \left({\mathbf 1}_{[a,a+\sigma \pi]}(x-t) - {\mathbf 1}_{[a+\sigma\pi,a+2\sigma\pi]}(x-t)\right) \widetilde{h}(t)
$$
where $\widetilde{h}$ denotes a control input  independent of $x$.
Then $b(x-t) h(x,t)= \widetilde{b}(x-t) \widetilde{h}(t)$ where
$$
\widetilde{b}(x)={\mathbf 1}_{[a,a+\sigma \pi]}(x) - {\mathbf 1}_{[a+\sigma\pi,a+2\sigma\pi]}(x)
$$
satisfies $ \int_{\T } \widetilde{b}(x)\, dx =0$. Moreover there exists by Lemma~\ref{lem:IrrationalInterval} (see below) a number
$C>0$ such that for all $k\in\Z^*$
$$
\widetilde{\beta}_k=\int_{\T} \widetilde{b}(x) e^{-ikx} dx\geq \frac{C}{|k|^3}.
$$
According to Theorem~\ref{thm1} we can find $\widetilde h\in L^2(\epsilon,T)$ steering $y(.,\epsilon)$ and
$\xi(.,\epsilon)$ to $y(.,T)=\xi(.,T)=0$ as soon as
$$
\sum_{k\neq 0} \frac{k^6 |\widetilde{c}_k| + k^4 |\widetilde{d}_k|  }{|\widetilde{\beta}_k|} < \infty,
$$
with
$$
y(x,\epsilon)=\sum_{k\in \Z} \widetilde{c}_k e^{ikx}, \quad \xi(x,\epsilon) =\sum_{k\in \Z} \widetilde{d}_k e^{ikx}.
$$
Let $W$ denote the space of the
couples $(\hat y,\hat \xi)\in L^2(\T )^2$ such that $||(\hat y , \hat \xi)||_W:= |\hat c_0| + |\hat d_0| +   \sum_{k\neq 0} (|k|^9 |\hat c_k| + |k|^7 |\hat d_k|) < \infty$, where
$\hat y(x)=\sum_{k\in \Z} \hat {c}_k e^{ikx}$ and $\hat \xi(x) =\sum_{k\in \Z} \hat{d}_k e^{ikx}$. Clearly, $W$ endowed with the norm $||\cdot ||_W$, is a Banach space.
Standard estimations based on the spectral decomposition used to prove Proposition~\ref{prop:cauchy} show that if the initial value $(y_0,\xi _0)$ lies in $W$, then the solution
of  \eqref{cont1}-\eqref{cont1bis} (with $h\equiv 0$)
remains in $W$. Therefore,  since $\sum_{k\neq 0} ( |k|^9 |c_k| + |k|^7 |d_k| ) < \infty$ and since the control is $C^\infty$ with respect to  $x\in\T$ and
$t\in[0,\epsilon]$, we also have $\sum_{k\neq 0} |k|^9  (|\widetilde{c}_k| + |k|^7 |\widetilde{d}_k|) < \infty$~(see e.g.~\cite{CazenH1998,Pazy1983}). Since
$(y_0,\xi _0)\in H^{s+2}(\T )\times H^s(\T ) $ with $s > 15/2$, we have by Cauchy-Schwarz inequality for  $\varsigma= 2s-15 >0$ that
$$
\sum_{k\neq0}  ( |k|^9 |c_k| + |k|^7 |d_k| )  \leq 2  (\sum_{k\neq 0} |k|^{-1-\varsigma} )^{\frac{1}{2}}
(\sum_{k\neq 0} |k|^{19+\varsigma} |c_k|^2 + |k|^{15+\varsigma} |d_k|^2 )^{\frac{1}{2}} < \infty.\qed
$$

\begin{lemma}\label{lem:IrrationalInterval}
   Let $\sigma \in(0,1)$ be a quadratic irrational, 
   and let $\tilde b$, $\tilde\beta_k$ be defined as above.
   Then $\tilde\beta_0=0$ and there exists $C>0$ such that for all $k\in\Z^*$, $|\tilde\beta_k| \geq \frac{C}{|k|^3}$.
\end{lemma}
 \begin{proof}
     Being a quadratic irrational, $\sigma$ is approximable by rational numbers to order 2 and to no higher order \cite[Theorem 188]{Hardy}); i.e., there exists
     $C_0>0$ such that for any  integers $p$ and $q$, $q\neq 0$, $\left|\sigma - \frac{p}{q} \right| \geq \frac{C_0}{q^2}$.
     On the other hand, $|\tilde\beta_k|=\frac{4}{|k|}\sin^2(\frac{\pi}{2}k\sigma)$ for $k\neq 0$. Pick any $k\ne 0$, take $p\in\mathbb Z$ such that $0\leq \frac{\pi}{2}k\sigma - p\pi <\pi$
     and use the elementary inequality $\sin^2\theta \geq \frac{4\theta^2}{\pi^2}$ valid for  $\theta\in[-\frac{\pi}{2},\frac{\pi}{2}]$. Then two cases occur.
  \begin{enumerate}
     \item[(i)] If $0\leq\tfrac{\pi}{2}k\sigma - p\pi \leq\tfrac{\pi}{2}$, then
  \[ \sin^2(\tfrac{\pi}{2}k\sigma)=
  \sin^2(\tfrac{\pi}{2}k\sigma-p\pi)\geq \tfrac{4}{\pi^2}(\tfrac{\pi}{2}k\sigma-p\pi)^2 = k^2
  \left( \sigma - \tfrac{2p}{k}\right)^2 \geq \tfrac{C_0^2}{k^2} ;\]
      \item[(ii)] If $-\tfrac{\pi}{2}\leq\tfrac{\pi}{2}k\sigma - (p+1)\pi \leq0$, then
  \[ \sin^2(\tfrac{\pi}{2}k\sigma-(p+1)\pi)\geq \tfrac{4}{\pi^2}(\tfrac{\pi}{2}k\sigma-(p+1)\pi)^2 = k^2
  \left( \sigma - \tfrac{2(p+1)}{k}\right)^2 \geq \tfrac{C_0^2}{k^2}. \]
  \end{enumerate}
 The lemma follows with $C=4C_0^2$.
\end{proof}

\subsection{Proofs of Theorem \ref{thm3} and Theorem \ref{thm4}} \label{ssec:ProofThm34}

The proofs are the same as for Theorem \ref{thm1}, with the obvious estimate
\[
\sum_{|k|>2} |k|^p(|k|^2|c_k|+|d_k|) \le C_{\varepsilon} \left( \sum_{|k|>2}\{  |k|^{2p+5+\varepsilon}|c_k|^2 + |k|^{2p+1+\varepsilon}|d_k|^2\right) ^{\frac{1}{2}}
\]
for $p\in \{ 3,4\}$, $\varepsilon >0$. \qed

\section{Proof of Proposition  \ref{prop1}} \label{Sec:Section3}

This section is devoted to the proof of Proposition \ref{prop1}. The method of proof is inspired from the one in \cite{FR,Glass,LK}. We first introduce
an entire function vanishing precisely at the $i\lambda _k^\pm$'s, namely the canonical product
\be
P(z)=z(1-\frac{z}{i\lambda _2})(1- \frac{z}{i\lambda _{-2}})\prod_{k\in \Z \setminus \{0, \pm 2\}} (1-\frac{z}{i\lambda _k^+})
\prod_{k\in\Z\setminus \{0,\pm 2\} } (1-\frac{z}{i\lambda _k^- })\cdot
\label{P1}
\ee
Next, following \cite{BM,Glass}, we construct a multiplier $m$ which is an entire function that does not vanish at the $\lambda _k^\pm$'s, such that
$P(z)m(z)$ is bounded for $z$  real while $P(z)m(z)$ has (at most) a polynomial growth in $z$ as $|z|\to \infty$ on each line $\text{Im } z=const.$
Next, for $k\in \Z \setminus \{ 0,\pm 2\}$ we construct a function $I_k^\pm$ from $P(z)$ and $m(z)$
and we define $\psi _k^\pm$ as the inverse Fourier transform of $I_k^\pm$. The other $\psi_k$'s are constructed in a quite similar way.
The fact that $\psi _k^\pm$ is compactly supported in time is a consequence of
Paley-Wiener theorem.

\subsection{Functions of  type sine}

To estimate carefully $P(z)$, we use the theory of functions of type sine (see e.g. \cite[pp. 163--168]{Levin} and \cite[pp. 171--179]{Young}).
\begin{definition}
\label{def1}
An entire function $f(z)$ of exponential type $\pi$ is said to be of {\em type sine} if
\begin{enumerate}
\item [(i)] The zeros $\mu _k$ of $f(z)$ are separated; i.e., there exists $\eta>0$ such that \\ $|\mu _k-\mu _l |\ge \eta\qquad k\ne l$;
\item[(ii)]   There exist positive constants $A,B$ and $H$ such that
\be
Ae^{\pi |y|} \le |f(x+iy)| \le Be^{\pi |y|}\quad \forall x\in\R,\ \forall y\in \R\ \  \text{with } |y|\ge H.\label{C1}
\ee
\end{enumerate}
\end{definition}
Some of the most important properties of an entire function of type sine are gathered in the following
\begin{proposition} (see \cite[Remark and Lemma  2 p. 164]{Levin}, \cite[Lemma 2 p. 172]{Young})
\label{propsine}
Let $f(z)$ be an entire  function of type sine, and let $\{ \mu _k \}_{k\in J}$ be the sequence of its zeros, where  $J\subset \Z$. Then
\begin{enumerate}
\item For any $\varepsilon >0$, there exist some constants  $C_\varepsilon , C_\varepsilon ' >0$ such that
\[
C_\varepsilon e^{\pi |\text{Im } z| }
\le  |f(z)| \le C_\varepsilon' e^{\pi |\text{Im }z|}\quad \text{ if } \text{dist}\{z,\{\mu _k\}  \} > \varepsilon.
\]
\item There exist some constants $C_1,C_2$ such that
$$
0<C_1 <|f ' (\mu _k) |<C_2<\infty \quad \forall k \in J.
$$
\end{enumerate}
\end{proposition}
Finally, we shall need the following result.
\begin{theorem} (see \cite[Corollary p. 168 and Theorem 2 p. 157]{Levin}
\label{thmsine}
Let $\mu _k=k+d_k$ for $k\in\Z$, with $\mu _0=0$, $\mu_k\ne 0$ for $k\ne 0$,  and $(d_k)_{k\in \Z}$ bounded, and let
\[
f(z)=z\prod_{k\in \Zs} (1-\frac{z}{\mu _k})=\lim_{K\to \infty } z\prod_{k\in \{-K,...,K\} \setminus \{ 0 \}  } (1-\frac{z}{\mu _k})\cdot
\]
Then $f$ is a function of type sine if, and only if, the following three properties are satisfied:
\begin{enumerate}
\item[(1)] $\inf_{k\ne l} |\mu_k-\mu _l |  >0$;
\item[(2)] There exists some constant $M>0$ such that
\[ \left\vert \sum_{k\in \Z} (d_{k+\tau} -d_k)\frac{k}{k^2+1}\right\vert \le M\quad \forall \tau \in \Z; \]
\item[(3)] $\displaystyle \limsup_{y\to +\infty} \frac{\log |f(iy)|}{y} =\pi ,\quad \limsup_{y\to -\infty} \frac{\log |f(-iy)|}{|y|} =\pi$.
\end{enumerate}
\end{theorem}
\begin{corollary}
\label{cor1}
Assume that $\mu_k=k+d_k$, where $d_0=0$ and $d_k=d+O(k^{-1})$ as $|k|\to \infty$  for some constant $d\in \C$, and that $\mu_k\ne \mu _l$ for
$k\ne l$. Then $f(z)=z\prod_{k\in \Zs} (1-\frac{z}{\mu _k})$ is an entire function of type sine.
\end{corollary}
\begin{proof}
We check that the conditions (1), (2) and (3) in Theorem \ref{thmsine} are fulfilled. \\
(1) From $\mu _k -\mu _l =k-l +O(k^{-1}, l^{-1})$ and the fact that $\mu_k-\mu _l \ne 0$ for $k\ne l$, we infer that (1) holds.\\
(2) Let us write $d_k=d+e_k$ with $e_k=O(k^{-1})$. Then for all $\tau\in \Z$
\[
\left(  \sum_{k\in \Z} (d_{k+\tau } -d_k)^2\right) ^{\frac{1}{2}} \le 2 \left(\sum_{k\in \Z } |e_k|^2\right) ^{\frac{1}{2}} <\infty .
\]
Therefore, for any $\tau \in \Z$, by Cauchy-Schwarz inequality
\begin{eqnarray*}
\left\vert
\sum_{k\in \Z} (d_{k+\tau} -d_k)\frac{k}{k^2+1}
\right\vert
&\le& \left(\sum_{k\in \Z}  (d_{k+\tau} -d_k)^2\right)^{\frac{1}{2}} \left(\sum_{k\in \Z} (\frac{k}{k^2+1} )^2 \right)^\frac{1}{2} \\
&\le& 2\left(\sum_{k\in \Z}|e_k|^2\right)^{\frac{1}{2}} \left(\sum _{k\in \Z} (\frac{k}{k^2+1})^2 \right)^{\frac{1}{2}}  =:M <\infty.
\end{eqnarray*}
(3) We first notice that
\[
f(z) = z\prod_{k=1}^\infty (1-\frac{z}{\mu_k })(1-\frac{z}{\mu_{-k}}) \cdot
\]
Let $z=iy$, with $y\in \R$. Then
\[
(1-\frac{z}{\mu _k})(1-\frac{z}{\mu _{-k}}) = 1 -\frac{y^2 +i(\mu_k +\mu _{-k}) y}{\mu _k\mu_{-k}}
\]
with
\[
\mu_k\mu_{-k} = -k^2 + O(k),\quad \mu_k + \mu_{-k} =2d + O(k^{-1}).
\]
It follows that for any given $\varepsilon \in (0,1)$, there exist $k_0\in \N ^*$ and some  numbers $C_1,C_2>0$ such that
\be
\label{Z1}
1+\frac{(1-\varepsilon) y^2-C_1 |y| }{|\mu_k\mu_{-k}|}
\le \left\vert   (1-\frac{z}{\mu_k})(1-\frac{z}{\mu_{-k}})  \right\vert \le 1+\frac{y^2+C_2|y| }{ | \mu_k\mu_{-k} | }
\ee
for $y\in \R$ and $k\ge k_0>0$.
Let
$$n(r):=\# \{ k\in \N ^*; \ |\mu_k\mu_{-k}|\le r \}.$$
Since $|\mu_k\mu_{-k} | \sim k^2$ as $k\to\infty$ and $\mu _k\ne 0$ for $k\ne 0$, we obtain  that
\ba
\sqrt{r} - C_3 \le n(r) \le \sqrt{r} +C_3\quad &&\text{ for }  r> 0, \\
n(r)=0 \quad &&\text{ for } 0<r<r_0,
\ea
for some constants $C_3>0$, $r_0>0$.
It follows that
\begin{eqnarray*}
\limsup_{|y|\to \infty} \frac{\log | f (iy) | }{|y|} &\le& \limsup _{|y|\to +\infty}
|y|^{-1} \sum_{k=1}^\infty \log \left\vert  (1-\frac{iy}{\mu_k})(1-\frac{iy}{\mu_{-k}})  \right\vert\\
&\le& \limsup _{|y|\to \infty }|y|^{-1} \sum_{k=1}^\infty \log (1+\frac{y^2+C_2|y|}{|\mu_k\mu_{-k}| }),
\end{eqnarray*}
where we used the fact that
\[
\lim_{|y|\to \infty} |y|^{-1} \log \left\vert  1 - \frac{iy}{\mu_{\pm k}} \right\vert =0 \quad \text{ for } 1\le k\le k_0.
\]
On the other hand, setting $\rho=y^2+C_2|y|\ge 0$, we have that
\begin{eqnarray*}
\sum_{k=1}^\infty \log (1+\frac{\rho}{|\mu_k\mu_{-k}|})
&=& \int_0^\infty \log (1+\frac{\rho}{t})\, dn(t) \\
&=& \rho \int_0^\infty \frac{n(t)}{t(t+\rho)}\, dt \\
&=& \int_0^\infty \frac{n(\rho s)}{s(s+1)}\, ds\\
&\le& \sqrt{\rho} \int_0^\infty \frac{ds}{\sqrt{s}(s+1)}+C_3\int_{r_0/\rho}^\infty \frac{ds}{s(s+1)}\\
&\le& |y|\sqrt{1+C_2|y|^{-1}}\pi +C_3 \log \left( 1+r_0^{-1} (y^2+C_2|y|) \right) .
\end{eqnarray*}
Thus
\[
\limsup_{|y|\to \infty}  \frac{\log | f(iy) | }{ |y| } \le \pi.
\]
Using again \eqref{Z1}, we obtain by the same computations that
\[
\limsup_{y\to +\infty} \frac{\log | f (iy) |}{y}\ge \pi, \ \text{ and } \limsup_{y\to -\infty} \frac{\log | f (iy) | }{ |y| } \ge \pi.
\]
The proof of (3) is completed.
\end{proof}
In what follows, $\text{arg } z$ denotes the principal argument of any complex number $z\in \C \setminus \R ^-$; i.e., $\text{arg }z\in (-\pi, \pi )$, and
\[
\log z = \log |z| +i\, \text{arg } z,\quad  \sqrt{z} = \sqrt{|z|} \, e^{i\frac{\text{arg } z}{2}}.
\]

We introduce, for $k\in \Zs$,
\[
\mu_k= \text{sgn} (k) \sqrt{ -\lambda _k^-} =k \sqrt{ \frac{1+\sqrt{1-4k^{-2} }}{2}  -ik^{-1}}=:k+d_k, \ k\in \Z
\]
with
\[
d_k=-\frac{i}{2} +O(k^{-1}).
\]
and $\mu _0=0$.
Let
\begin{eqnarray}
P_1(z) &=&z\prod_{k\in \Zs}  (1+ \frac{z}{i\lambda _k^+}  ), \label{C2a}\\
P_2(z) &=&z\prod_{k\in \Zs}  (1+  \frac{z}{i\lambda _k^-}  ),   \label{C2b}\\
P_3(z) &=& z^2 \prod _{k\in \Zs} (1+ \frac{z^2}{\lambda_k^-}   ), \label{C2c}\\
\text { and } \qquad P_4(z) &=& z\prod _{k\in \Z \setminus \{ 0 \} }  (1- \frac{z}{\mu_k }  ).  \label{C2d}
\end{eqnarray}
It follows from \eqref{A17} that the convergence in \eqref{C2b} is uniform in $z$ on each compact set of $\C$, so that
$P_2$ is an entire function. Note also that
\ba
P_2(z) &=& iP_3(e^{-i\frac{\pi}{4}} \sqrt{z} ), \label{C5a}\\
P_3(z) &=& -P_4(z)P_4(-z), \label{C5b} \\
P(z)     &=&\frac{ P_1(-z) P_2(-z) }{ z (1-\frac{z}{i\lambda _2} )( 1-\frac{z}{i\lambda _{-2}} )} \cdot \label{Q1}
\ea
Applying Corollary \ref{cor1} to $P_1$, noticing that
\[
-i\lambda_k^+ =k+i+O(k^{-2})
\]
with $\lambda_k^+\ne \lambda_l^-$ for $k\ne l$, and $\lambda _0^+=0$, we infer that $P_1(z)$ is an entire function of sine type.
Thus, for given $\varepsilon >0$ there are some positive constants $C_4,C_5,C_6$ such that
\ba
C_4e^{\pi |y|} \le  |P_1(x+iy)|   &\le& C_5 e^{\pi |y|}, \quad  \text{dist } (x+iy, \{ -i\lambda _k^+ \} ) > \varepsilon   \label{C6}\\
|P_1'(-i\lambda_k^+)|   &\ge& C_6, \quad k\in \Z.  \label{C7}
\ea
Next, applying Corollary \ref{cor1} to $P_4$, noticing that
\[
\mu_k=k-\frac{i}{2} +O(k^{-1})
\]
with $\mu_k\ne \mu_l$ if $k\ne l$ and $\mu_0=0$, we infer that $P_4(z)$ is also an entire function of sine type. In particular, it is of
exponential type $\pi$
\be
|P_4(z) | \le C e^{\pi |z|},\qquad z\in \C.
\ee
Therefore, we have for any $\varepsilon >0$ and for some positive constants $C_7,C_8,C_9$
\ba
C_7e^{\pi |y|}  \le |P_4(x+iy)| &\le &C_8 e^{\pi |y|},\quad \text{dist } (x+iy,\{\mu _k\})>\varepsilon  \label{C8}\\
|P_4'(\mu _k)| &\ge& C_9,\quad k\in \Z. \label{C9}
\ea
In particular, $P_3$ is an entire function of exponential type $2\pi$ with
\be
C_7^2  e^{2\pi |y|}  \le |P_3(x+iy)| \le C_8^2 e^{2\pi |y|} \qquad \text{dist }(\pm ( x+iy) ,\{\mu _k  \} ) >\varepsilon \label{C10}.
\ee
Combined to \eqref{C5a}, this yields
\be
|P_2(z)|\le C e^{2\pi \sqrt{|z|}}\qquad z\in \C.
\ee
Substituting $e^{-i\frac{\pi}{4}}\sqrt{z}$ to $x+iy$ in \eqref{C10} yields
\be
C_7^2 \exp( 2\pi |\text{Im} (e^{-i\frac{\pi}{4}}\sqrt{z})| )
\le |P_2(z)| \le
C_8^2 \exp( 2\pi |\text{Im} (e^{-i\frac{\pi}{4}}\sqrt{z} )| )\quad
\text{dist} (\pm e^{-i\frac{\pi}{4}}\sqrt{z},\{ \mu _k \} ) >\varepsilon.
\label{C12}
\ee
From \eqref{C12} (applied for $x$ large enough)  and the continuity of $P_2$ on $\C$, we obtain that
\be
|P_2(x)|  \le C e^{\sqrt{2}\pi \sqrt{ |x| } }.\label{Q2}
\ee
We are now in a position to give bounds for the canonical product $P$ in \eqref{P1}.
\begin{proposition}
\label{prop3}
The canonical product $P$ in \eqref{P1} is an entire function of exponential type at most  $\pi$. Moreover, we have
for some constant $C>0$
\ba
|P(x)| &\le & C (1+ |x|) ^{-3}  e^{\sqrt{2}\pi \sqrt{|x|} }, \qquad x\in \R , \label{C13}\\
|P'( i \lambda _k^+ )| &\ge&  C^{-1} |k|^{-3} e^{\sqrt{2} \pi \sqrt{ |k| }  }  \qquad k\in \Z\setminus \{ 0,\pm 2 \}, \label{C14a}\\
|P'(i\lambda _k^- ) |   &\ge&  C^{-1} |k|^{-7} e^{\pi k^2} \qquad  k\in \Z \setminus \{ 0, \pm 2 \}.\label{C14b}
\ea
\end{proposition}

\begin{proof}
Note first that $\text{dist} (\R , \{-i \lambda _k^+; \ k\ne 0   \} )>0$ from \eqref{ZZ1}.
Since $(1+\frac{is}{z})P_1(z)$ is
also an entire function of sine type for $s\gg 1$, with $\text{dist} (\R , \{ -i\lambda _k^+;\ k\ne 0\} \cup \{ is \})>0$,
we infer from Proposition  \ref{propsine}  that for some constant $C>0$
\[
|P_1(x)| \le C \qquad \forall x\in \R .
\]
Combined to \eqref{Q1} and \eqref{Q2}, this yields \eqref{C13}. Let us turn to \eqref{C14a}. Note first that for $k\in \Z\setminus \{ 0,\pm 2 \}$
\be
P'(i\lambda _k^+) =  P_1'(-i\lambda _k^+) \frac{ P_2(-i\lambda _k^+)}{ (-i\lambda _k^+) (1-\frac{\lambda_k^+}{\lambda _2} ) (1-\frac{\lambda_k^+}{\lambda_ {-2} }) }
\cdot \label{C15}
\ee
Clearly, for some $\delta  >0$, $|\lambda _k^+-\lambda _l^-|>\delta   $ for all $k\in \Z \setminus \{ 0, \pm 2\}$, $l\in \Z$, and
\[
|\text{Im }(e^{-i\frac{\pi}{4}} \sqrt{-i\lambda _k^+})| =
|\text{Im } \big( \frac{1-i}{\sqrt{2} } \sqrt{k+i+O( k^{-2} ) }\big) | = \sqrt{ \frac{|k|}{2}}  + O(|k|^{-\frac{1}{2}}).
\]
With \eqref{C12}, this gives
\be
|P_2(-i\lambda _k^+)| \ge Ce^{\sqrt{2} \pi \sqrt{|k|} }. \label{C17}
\ee
It follows then from \eqref{C7}, \eqref{C15}, and \eqref{C17} that
\[
|P'( i\lambda _k^+)| \ge C\frac{e^{ \sqrt{2}\pi \sqrt{|k|}}}{|k|^3}
\]
for some constant $C>0$ independent of $k\in \Z \setminus \{ 0 , \pm  2\}$.
On the other hand
\be
P'( i\lambda_k^-) =  P_2'(-i\lambda _k^-)
\frac{ P_1(-i\lambda _k^-)}{ (-i\lambda _k^-) (1-\frac{\lambda_k^-}{\lambda _2} ) (1-\frac{\lambda_k^-}{\lambda_ {-2} }) }
\cdot \label{C18}
\ee
By \eqref{A17} and \eqref{C6}, we have that
\[
|P_1(-i\lambda _k^-)| \ge Ce^{\pi k^2 }, \qquad k\in \Z \setminus \{ 0, \pm 2\} \cdot
\]
From \eqref{C5a}-\eqref{C5b}, we have that
\[
P_2'(z) =\frac{e^{ i \frac{\pi}{4}} } {2\sqrt{z}}
\left[  P_4' (e^{-i\frac{\pi}{4}} \sqrt{z})P_4(-e^{-i\frac{\pi}{4}} \sqrt{z})
-  P_4 (e^{-i\frac{\pi}{4}} \sqrt{z})P_4'(-e^{-i\frac{\pi}{4}} \sqrt{z})
\right] .
\]
For $z=-i\lambda _k^-$, $e^{-i\frac{\pi}{4}}\sqrt{z}=\sqrt{-\lambda _k^-} = \text{sgn}\,  (k) \mu_k$, hence
\[
P_2'(-i\lambda_k^- ) = \frac{1}{2\mu _k} P_4'(\mu _k) P_4 (- \mu _k) .
\]
Since $|\mu _k +\mu _l | > \delta >0$  for $k\in \Zs$ and $l\in \Z$, we have from \eqref{C8} that $|P_4(-\mu _k )|\ge c$ while, by \eqref{C9},
$|P_4'(\mu _k)|>c>0$. It follows that for some constant $C>0$
\[
|P_2'(-i\lambda _k^-)| \ge \frac{C}{|k|}\qquad \forall k\in \Zs.
\]
Therefore,
\[
|P'(i\lambda _k^-)| \ge C\frac{e^{\pi k^2}}{|k|^7},\qquad k\in \Zs.
\]
\end{proof}
We seek for an entire function $m$ (the so-called {\em multiplier})  such that
\begin{eqnarray*}
|m(x)| &\le& C(1+|x|)e^{-\sqrt{2} \pi \sqrt{|x|}}, \qquad x\in \R ,\\
|m(i\lambda _k^+)| &\ge & C^{-1} |k|^{-3} e^{-\sqrt{2} \pi \sqrt{|k|}},\qquad k\in \Z \setminus \{ 0\} , \\
|m(i \lambda _k^- )|  &\ge&  C^{-1} e^{a\pi k^2 -2\sqrt{2} \pi \sqrt{ |k |} }, \qquad k\in \Z \setminus \{ 0\} .
\end{eqnarray*}
We shall use the same multiplier as in \cite{Glass}, providing additional estimates required to evaluate it at the points
$i\lambda _k^-$ for $k\in \Z$.
Let
\be
s (t)=at - b\sqrt{t},\qquad t>0 \label{D1}
\ee
where the constants $a>0$ and $b>0$ will be chosen later.
Note that $s$ is increasing for $t>\left(\frac{b}{2a} \right) ^{2}$ and that $s(B)=0$ where $B=(b/a)^2$. Let
\be
\nu (t) = \left\{
\begin{array}{ll}
0 \quad &t\le B,\\
s(t) & t\ge B.
\end{array}
\right.
\label{D2}
\ee
Introduce first
\ba
g(z) &=& \int_0^\infty \log (1-\frac{z^2}{t^2} ) d\nu (t) = \int_B^\infty \log (1-\frac{z^2}{t^2}) ds (t) \qquad z\in \C \setminus \R , \label{D3}\\
U(z) &=& \int_0^\infty \log |1-\frac{z^2}{t^2}| d\nu (t) = \int_B^\infty \log |1-\frac{z^2}{t^2}|  ds(t) \qquad z\in \C .\label{D4}
\ea
Note that $g$ is holomorphic on $\C \setminus \R$ and $U$ is continuous on $\C$, with $U(z)=\text{Re } g(z)$. Next we atomize
the measure $\mu$ in the above integrals, setting
\ba
\tilde g(z) &=& \int_0^\infty \log (1-\frac{z^2}{t^2} ) d [\nu (t)]\qquad z \in \C \setminus \R  , \label {D5} \\
\tilde U (z) &=& \int_0^\infty \log \vert 1-\frac{z^2}{t^2} \vert  d[\nu (t) ] \qquad z\in \C ,\label{D6}
\ea
where $[x]$ denotes the integral part of $x$.
Again, $\tilde g$ is holomorphic on $\C \setminus \R$ and $\tilde U$ is continuous on $\C$ with $\tilde U(z)=\text{Re } g(z)$. Actually,
$\exp \tilde g$ is an entire function. Indeed, if $\{ \tau _k\} _{k\ge 0} $ denotes the sequence of discontinuity points for $t\mapsto [\nu (t)]$,
then $\tau _k\sim  k/a$ as $k\to \infty$ and
\be
\tilde g(z) =\sum_{k\ge 0} \log (1-\frac{z^2}{\tau _k^2}), \quad z\in \C \setminus \R \cdot
\label{D7}
\ee
Therefore,
\be
e^{\tilde g(z) }  = \prod_{k\ge 0} (1-\frac{z^2}{\tau _k^2}),
\label{D8}
\ee
the product being uniformly convergent on any compact set in $\C$.
We shall pick later $m(z) =\exp (\tilde g ( z- i ))$ with $a=\frac{T}{2\pi} - 1$ and $b=\sqrt{2}$.
The strategy, which goes back to \cite{BM}, consists in estimating carefully $U$, and next $U-\tilde U$.
Let for $x>0$
\be
w (x)=-\pi \sqrt{x} + x\log \left\vert \frac{x+1}{x-1} \right\vert -
\sqrt{x} \log \left\vert  \frac{\sqrt{x}+1}{\sqrt {x} -1}  \right\vert
+2\sqrt{x} \arctan ( \sqrt{x} )  .
\label{D9}
\ee
Note that $w\in L^\infty (\R ^+)$, for $\lim_{x\to \infty} w (x) =-2$ and  $w(0^+)=0$.
\begin{lemma}\label{lemglass}
\cite{Glass} It holds
\be
\label{D10}
U(x) +b\pi \sqrt{|x|} = -a Bw(|x|)\qquad \forall x\in \R
\ee
\end{lemma}
Our first aim is to extend that estimate to the whole domain $\C$.
\begin{lemma}
\label{lem3}
There exists some positive constant $C=C(a,b)$ such that
\begin{equation}
 -C-b\pi (1+\frac{1}{\sqrt{2}} ) \sqrt{ |y| }  \le U(z)  + b\pi \sqrt{|x|} -a \pi  |y| \le C,\qquad z=x+iy\in \C.  \label{D16}
\end{equation}
\end{lemma}
\begin{proof}
We follow the same approach as in \cite{Glass}. We first use the following identity from \cite[(36)]{Glass} (note that $U$ is even)
\be
U(z)= |\text{Im } z | (\pi a + \frac{1}{\pi} \int_{-\infty}^\infty \frac{U(t) }{|z-t|^2}dt).
\label{D18}
\ee
To derive \eqref{D16}, it remains to estimate the integral term in \eqref{D18} for $z=x+iy \in \C$.
We may assume without loss of generality that $y>0$. From Lemma \ref{lemglass}, we can write
\[
U(t) = -b\pi \sqrt{ |t| } -aB w( |t| )
\]
where $w\in L^\infty (\R ^+)$. Then, with $t=ys$,
\begin{eqnarray}
\left\vert \frac{y}{\pi} \int_{-\infty}^\infty \frac{aB w( |t| ) }{(x-t)^2 + y^2}\,  dt \right\vert
&\le&  ||w||_{L^\infty ( \R ^+)} \frac{aBy}{\pi} \int_{-\infty}^\infty \frac{ds}{ y ( (\frac{x}{y}-s)^2 +1)} \nonumber \\
&=& aB ||w||_{L^\infty (\R ^+ ) }=:C. \label{Z10}
\end{eqnarray}
On the other hand, still with $t=y s $, and using explicit computations in \cite{Glass} of some integral terms,
\begin{eqnarray*}
\frac{y}{\pi} \int_{-\infty}^\infty (-b\pi) \frac{\sqrt{|t|}}{(x-t)^2 +y^2} dt
&=& -b\sqrt{y} \int_{-\infty}^\infty \frac{\sqrt{|s|}}{(\frac{x}{y}-s)^2 +1}ds  \\
&=& -b\sqrt{y} \left( \frac{\pi}{\sqrt{2\sqrt{ 1+\frac{x^2}{y^2}} -\frac{2x}{y}}}
+  \frac{\pi}{\sqrt{2\sqrt{ 1+\frac{x^2}{y^2}} +\frac{2x}{y}}}
\right)\\
&=& -b\frac{\pi}{\sqrt{2}} ( \sqrt{\sqrt{x^2+y^2 } -x } + \sqrt{\sqrt{x^2+y^2}+x} ).
\end{eqnarray*}
Routine computations give
\[
\sqrt{2|x|} \le \sqrt{\sqrt{x^2+y^2} -x } + \sqrt{\sqrt{x^2+y^2}+x} \le \sqrt{2|x|} +(\sqrt{2} +1)\sqrt{y}\quad
\forall x\in \R,\ \forall y>0.
\]
Therefore
\[
-b\pi \sqrt{ |x| }  - b \pi (1+ \frac{1}{\sqrt{2}} ) \sqrt{y} \le \frac{y}{\pi}
\int_{-\infty}^\infty (-b\pi) \frac{ \sqrt{|t|}}{(x-t)^2+y^2} dt
\le -b\pi \sqrt{|x|}.
\]
Combined to \eqref{D18} and \eqref{Z10}, this yields \eqref{D16}.
\end{proof}
In order to obtain estimates for $\tilde{U} (z)$, we need to give bounds from above and below for
\[
\tilde{U} (z) - U(z) =\int_0^\infty \log |1-\frac{z^2}{t^2} |\,  d( [\nu ](t) - \nu (t) ).
\]
We need the following lemma, which is inspired from \cite[Vol. 2, Lemma p. 162]{Koosis}
\begin{lemma}
\label{lem4}
Let $\nu:\R ^+\to \R ^+$ be nondecreasing and null on $(0,B)$. Then for $z=x+iy$ with $y\ne 0$, we have
\be
-\log ^+ \frac{|x|}{|y|} -\log ^+ \frac{x^2+y^2}{B^2} -\log 2 \le I =
\int_0^\infty \log |1-\frac{z^2}{t^2} | d ([\nu ] (t) -\nu (t)) \le \log ^+ \frac{|x|}{|y|} \cdot
\label{D26}
\ee
\end{lemma}
\begin{proof}
The proof of the upper bound is the same as in \cite{Koosis}. It is sketched here just for the sake of completeness.
Pick any $z=x+iy$ with $y\ne 0$. Integrate by part in $I$ to get
\[
I=\int_0^\infty (\nu (t) - [\nu (t)] ) \frac{\partial }{\partial t } \log |1-\frac{z^2}{t^2} | dt. \]
Let $\zeta=\frac{z^2}{t^2}$. If $\text{Re } z^2 \le 0$ (i.e. if $|x|\le |y|$), then the distance $|1-\zeta |$ is decreasing w.r.t. $t$
($t\in (0,+\infty)$),
so that $I\le 0$. If $\text{Re } z^2 >0$, then $|1-\zeta |$ decreases to the minimal value
$\frac{ |\text{Im } z^2|}{|z^2|}$ taken at $t=t^*:=\frac{|z|^2}{|x|}$, and then it increases. Since $0\le \nu (t) - [ \nu (t) ] \le 1$, we have that
\[
I \le \int_{t^*}^\infty \frac{\partial}{\partial t} \log |1-\frac{z^2}{t^2} | dt = \log \frac{|z^2|}{|\text{Im } z^2|}
=\log ( \frac{|x|}{2|y|} + \frac{|y|}{2|x|}) \le \log \frac{|x|}{|y|}   \cdot
\]
Let us pass to the lower bound. If $\text{Re } z^2\le 0$,
\[
I\ge \int_B^\infty \frac{\partial }{\partial t} \log |1-\frac{z^2}{t^2} | dt = -\log |1-\frac{z^2}{B^2} | \cdot
\]
Assume now that $\text{Re } z^2 >0$. If $t^*=\frac{|z^2|}{|x|} \le B$, $I\ge 0$. If $t^*>B$, then
\[
I\ge \int_B^{t^*}\frac{\partial}{\partial t} \log |1-\frac{z^2}{t^2}| dt =-\log \frac{|z^2|}{|\text{Im } z^2|}
-\log  |1-\frac{z^2}{B^2}|\cdot
\]
Note that
\[
\log | 1-\frac{z^2}{B^2}| \le \log (1+|\frac{z}{B}| ^2) \le \log ^+ \frac{x^2+y^2}{B^2} + \log 2.
\]
Therefore
\[
I\ge -\log ^+ \frac{|x|}{|y|} -\log ^+ \frac{x^2+y^2}{B^2} -\log 2.
\]
\end{proof}
Gathering together Lemma \ref{lem3} and Lemma \ref{lem4}, we obtain the
\begin{proposition}
\label{prop4}
There exists some positive constant $C=C(a,b)$ such that for any complex number $z=x+iy$ with $y\ne 0$,
\be
-C-b\pi (1+\frac{1}{\sqrt{2}} )\sqrt{|y|}  -\log^+ \frac{|x|}{|y|} -\log ^+ (\frac{x^2+y^2}{B^2}) -\log 2
\le \tilde U (z) +b\pi \sqrt{|x|} -a\pi |y| \le C + \log ^+ \frac{|x|}{|y|}. \label{D30}
\ee
\end{proposition}
Pick now
\be
a=\frac{T}{2\pi} -1>0, \quad b=\sqrt{2}, \quad \text{ and } m(z) = \exp \tilde g (z-i)
\label{D31}
\ee
Note that $|m(z)|=\exp \tilde U (z-i)$.
The needed estimates for the multiplier $m$ are collected in the following
\begin{proposition}
\label{prop5}
$m$ is an entire function on $\C$ of exponential type at most $a\pi$. Furthermore, the following estimates hold for some constant $C>0$:
\ba
| m(x) |                              &\le  & C(1+|x|) e^{-\sqrt{2} \pi \sqrt {|x|} }, \qquad x\in \R \label{D32}\\
| m(i \lambda _k^+) |     &\ge & C^{-1}    |k|^{-3} e^{-\sqrt{2}\pi \sqrt{|k|}} ,\qquad k\in \Zs \label{D33} \\
| m(i \lambda _k^- )  |     &\ge & C^{-1}    e^{a\pi k^2 -2\sqrt{2} \pi |k|} , \qquad k\in \Zs . \label{D34}
\ea
\end{proposition}
\begin{proof}
\eqref{D32} follows at once from \eqref{D30} (with $y= - 1$). We infer from \eqref{ZZ1} that for $k\in \Zs$
\be
\text{Im } (i\lambda _k^\pm ) \le -\frac{1}{2}.
\label{FF1}
\ee
It follows then  from \eqref{A16} and \eqref{D30} that
\[
|m(i\lambda _k^+) | =\exp \tilde U ( - k- 2i (1+O (k^{-2}))) \ge C |k|^{-3} e^{-\sqrt{2} \pi \sqrt{|k|} } \quad (k\ne 0).
\]
Finally, from \eqref{A17} and \eqref{D30}, we infer  that
\begin{eqnarray*}
|m (i \lambda _k^-) | &=& \exp \tilde U (-k-i(k^2 +O(k^{-2})) ) \\
&\ge & C\exp (-\sqrt{2} \pi \sqrt{|k|}  +a\pi k^2 -(\sqrt{2} + 1) \pi |k| -4\log |k| ) \\
&\ge & C\exp (a\pi k^2 -2\sqrt{2} \pi  |k|).
\end{eqnarray*}
\end{proof}
We are in  a position to define the functions in the biorthogonal family. Pick first any
$k\in \Z \setminus \{0,\pm 2\}$, and set
\[
I_k^{\pm} (z) = \frac{P(z)}{P'(i\lambda _k^\pm )(z - i\lambda _k^\pm)} \cdot \frac{m(z)}{m(i\lambda _k^\pm )}\cdot
\frac{ (1-\frac{z}{i\lambda _2})(1-\frac{z}{i\lambda _{-2}})}{(1-\frac{\lambda _k^\pm}{\lambda _2})(1-\frac{\lambda _k^\pm}{\lambda _{-2}})}\cdot
\]
Clearly, $I_k^\pm$ is an entire function of exponential type at most $\pi (1+a)=T/2$. Furthermore, we have that
\be
\label{G1}
I_k^\pm (i\lambda _l^\pm ) =\delta _k^l\delta _+^-\qquad \forall l\in \Z,
\ee
where $\delta _+^-$ is 1 if the two signs in the l.h.s. are the same, and 0 otherwise.
Moreover,
\be
\label{G2}
(I_k^\pm )' (i \lambda_{\pm 2})=0.
\ee
On the other hand, by \eqref{A16}, \eqref{C13}, \eqref{C14a}, \eqref{D32} and \eqref{D33},  we have that
\[
| I_k^+ (x)| \le C \frac{|k|^4 }{ | x - i\lambda _k^+  |} \le C \frac{|k|^4}{1+|k+x|} \cdot
\]
Thus $I_k^+ \in L^2(\R )$ with
\be
\label{F6}
||I_k^+ ||_{L^2(\R ) } \le C |k|^4.
\ee
Finally, by \eqref{A17}, \eqref{C13}, \eqref{C14b}, \eqref{D32} , and \eqref{D34}, we have that
\[
| I_k^-(x)| \le C \frac{|k|^3}{| x  - i\lambda _k^-  | } e^{-(a+1)\pi k^2 +2\sqrt{2} \pi |k| } \le C\frac{|k|^3}{|x+k|+k^2}
e^{-\frac{T}{2} k^2 +2\sqrt{2}\pi |k|}.
\]
Thus
\be
\label{F7}
|| I_k^- ||_{L^2(\R )} \le C |k|^2 e^{-\frac{T}{2} k^2 +2\sqrt{2}\pi |k|}.
\ee

It remains to introduce the functions $I_0(z), I_2(z),I_{-2}(z), \tilde I_2(z)$, and ${\tilde I}_{-2}(z)$. We set
\begin{eqnarray*}
I_0(z) &=& \frac{P(z)}{P'(0)z} \cdot \frac{m(z)}{m(0)}\cdot
(1-\frac{z}{i\lambda _2})(1-\frac{z}{i\lambda _{-2}} ),\\
\tilde I_2(z) &=& -i \frac{P(z)}{P'(i\lambda _2)}\cdot \frac{m(z)}{m(i\lambda _2)}\cdot \frac{1-\frac{z}{i\lambda _{-2}}}{1-\frac{\lambda _2}{\lambda_{-2}}},\\
\tilde I_{-2}(z) &=& -i \frac{P(z)}{P'(i\lambda _{-2} )}\cdot \frac{m(z)}{m(i\lambda _{-2})}\cdot
\frac{1-\frac{z}{i\lambda _{2}} }{1-\frac{\lambda _{-2}}{\lambda_2}},\\
K_2(z) &=&i \frac{\tilde I_2 (z) }{ z - i\lambda _2 },\quad I_2(z)=K_2(z) - iK_2'(i\lambda _2)\tilde I_2 (z),\\
K_{-2}(z) &=&i \frac{\tilde I_{-2} (z) }{ z   - i\lambda _{-2} },\quad I_{-2}(z)=K_{-2}(z) - iK_2'(i\lambda _{-2})\tilde I_{-2} (z).
\end{eqnarray*}

Then we have that
\begin{eqnarray}
&& I_0(0)=1,\quad  I_0(i\lambda _k^\pm ) =0 \quad k\in \Zs, \quad I_0' (i\lambda _{\pm 2} )=0, \label{F1}\\
&&\tilde I_2 (i\lambda _k^\pm )=0\quad k\in \Z, \quad  \tilde I_2 '(i\lambda _2) =-i,\quad \tilde I_2 '(i\lambda _{-2})=0,\label{F2}\\
&&\tilde I_{-2}(i\lambda _k^\pm )=0 \quad k\in \Z,\quad  \tilde I_{-2}'(i\lambda _{-2})=-i,\quad \tilde I_{-2}'(i\lambda _2)=0,\label{F3} \\
&&I_2(i\lambda _k^\pm) =0\quad  k\in\Z \setminus \{ 2 \},\quad I_2 (i\lambda _2 ) =1,\quad I_2'(i\lambda _{\pm 2}) = 0,\label{F4}\\
&&I_{-2}(i\lambda _k^\pm )=0\quad k\in \Z\setminus \{ -2 \},\quad  I_{-2}(i\lambda _{-2}) =1,\quad I_{-2}'(i\lambda _{\pm 2})=0.\label{F5}
\end{eqnarray}
Moreover, $I_0,\ \tilde I_2,\ \tilde I_{-2},\ I_2$, and $I_{-2}$ are entire functions of exponential type at most $\pi (1+a )$
and they belong all to $L^2(\R )$.

Let $\psi _k^\pm$, $\psi _k$, and $\tilde \psi _k$  denote the inverse Fourier transform of $I_k^\pm$, $I_k$, and $\tilde I_k$ for
$k\in \Z\setminus \{ 0,\pm 2\}$, $k\in \{0,\pm 2\}$ and $k\in \{ \pm 2\}$, respectively.
Then, by Paley-Wiener theorem, the functions $\psi_k^\pm$, $\psi_k$ and $\tilde \psi_k$ belong to $L^2(\R)$, and are supported in $[-T/2,T/2]$.
On the other hand, if $I(z)=\hat \psi (z)=\int_{-\infty}^\infty \psi (t) e^{-itz}dt$ with $\psi\in L^2(\R),\ \text{supp } \psi \subset [-T/2,T/2]$, then
\[
\int_{-\frac{T}{2}}^{\frac{T}{2}} \psi (t) e^{\lambda t} dt =  I(i\lambda )\quad \text{ and }  -i \int_{-\frac{T}{2}}^{\frac{T}{2}} t  \psi (t) e^{\lambda t }   dt =  I'(i\lambda ).
\]
Thus \eqref{E1}-\eqref{E5} follow from \eqref{G1}-\eqref{G2} and \eqref{F1}-\eqref{F5}, while \eqref{E6}-\eqref{E7} follow  from \eqref{F6}-\eqref{F7}.
The proof of Proposition \ref{prop1} is complete.

\section{Concluding remark}

In this paper, the equation  $y_{tt}-y_{xx}-y_{txx} = b(x-u(t))h(t)$ is proved to be null controllable on the torus (i.e. with periodic boundary conditions) when the support of the scalar control $h(t)$ moves at a constant velocity $c$  ($u(t)=ct$). What happens for a domain with boundary?
 More precisely, we may wonder under which assumptions on the initial conditions, the control time $T$, the support of the controller $b$  and its pulsations $\omega$
 the null controllability of the system
\begin{eqnarray*}
&&y_{tt}-y_{xx}-y_{txx} = b(x-\cos(\omega t))h(t), \qquad x\in (-1,1),~  t\in(0,T),\\
&&y(-1,t)=y(1,t)=0,\qquad t\in(0,T)
\end{eqnarray*}
holds.
\section*{Acknowledgements}

LR was  partially supported by the Agence Nationale de la Recherche, Project CISIFS,
grant ANR-09-BLAN-0213-02.

\end{document}